\documentclass{azjmPRE}

\theoremstyle{definition}

\theoremstyle{definition} 

\newtheorem*{maintheorem*}{Main Theorem}

\begin{document}


\title{Heyting Algebra and G\"{o}del Algebra vs. various Topological Systems and Esakia Space: a Category Theoretic Study}

\author[ A. Di Nola, R. Grigolia, P. Jana]{ A. Di Nola \corrauth, R. Grigolia, P. Jana}

\begin{abstract}
This paper introduces a notion of intuitionistic topological system. Properties of the proposed system is studied in details. Categorical interrelationships among Heyting algebra, G\"{o}del algebra, Esakia space and proposed intuitionistic topological systems have also been studied. A flavour of Kripke model is given.
\end{abstract}

\keywords{Heyting algebra, G\"{o}del algebra, Esakia space, Intuitionistic logic, Topological system}
\ams{06D20, 97H50, 03G25, 03B20, 54H10}

\maketitle
\section{Introduction}
Topological system was introduced by S. Vickers in his book ``Topology via Logic" \cite{SV} in 1989. A topological system is a triple $(X,\models, A)$, consisting of a non empty set $X$, a frame $A$ and a binary relation between the set and the frame, which matches the logic of finite observations or geometric logic. Topological system is a mathematical object which unifies the concepts of topological space and frame in one framework. Hence such a structure allows us to switch among the concepts of frame, topological space and corresponding logic freely.

Concepts of a topological system and geometric logic or logic of finite observations have a deep connection. It is well known that the Lindenbaum algebra of geometric logic is a frame, likewise the Lindenbaum algebra of classical logic is Boolean algebra and that of intuitionistic logic is Heyting algebra etc. One may notice that any topological system is a model of geometric logic. 

In \cite{SV1}, it may be noticed that ``Logically, spatiality is the same as completeness, but there is a difference of emphasis. Completeness refers to the ability of the logical reasoning (from rules and axioms) to generate all the equivalences that are valid for the models: if not, then it is the logic that is considered incomplete. Spatiality refers to the existence of enough models to discriminate between logically inequivalent formulae: if not, then the class of models is incomplete." In this respect we may recall that there exist adjunction between category of topological systems and the category of topological spaces, which leads to the concept that not every topological system comes from a topological space. To elaborate the fact one may notice that every topological space can be considered as a topological system because of the following fact: if $(X,\tau)$ is a topological space then $(X,\vdash,\tau)$ is the corresponding topological system, where $x\vdash T$ represents that $x$ is an element of $T(\in \tau$). Hence not every topological system is spatial and correspondingly we arrive at the conclusion (logical fact) that the corresponding logic (i.e., geometric logic) is not complete. On the contrary whenever we deal with a logic which is complete then we can expect categorical equivalence or duality between categories of mathematical structures which are the models of the logic.

Topological system is an important mathematical structure in its own right. It is already mentioned earlier that this kind of structure reflects the corresponding topological and algebraic structures simultaneously. In fact it is closely connected to the corresponding logic. On the other hand topological system plays important roles in computer science and (quantum) physics \cite{SV, DI}. 

It is well known that the category of Heyting algebras are dually equivalent to the category of Esakia spaces. Consequently both Heyting algebra and Esakia space are models of intuitionistic logic. Our main goal of this paper is to introduce a notion of I-topological system such that it will able to unify the notions of Heyting algebra, Esakia space and I-topological system in it. The similar study for G\"{o}del algebra and related structures is also a focus point for the present paper. It is quite expected that the proposed notions will have its impact in the areas of computer science and physics.

This paper is organised as follows. Section 2 contains the required preliminary notions to make the paper self contained. Notion of I-topological system is introduced and studied in details in Section 3. This section gives a cue to connect the proposed system with Kripke model. A detailed categorical study of the proposed systems with corresponding topological and algebraic structure has also done in this section. Section 4 contributes some concluding remarks.

\section{Preliminaries}
In this section we include a brief outline of relevant notions to develop our proposed mathematical structures and results. In \cite{AJ, LE, PJ, SV}, one may found the details of the notions stated here.
\begin{definition}[$G$-structured arrow and $G$-costructured arrow]\label{2.1} Let  $G:\mathbb{A}\to\mathbb{B}$ be a functor, where $\mathbb{A},\ \mathbb{B}$ are two categories and let $B$ be a $\mathbb{B}$-object. Then the concepts of $G$-structured arrow and $G$-costructured arrow are defined as follows.
\begin{enumerate}
\item  A $G$-$\mathbf{structured\ arrow\ with\ domain}$ \index{$G$-structured arrow} $B$ is a pair $(f,A)$ consisting of an $\mathbb{A}$-object $A$ and a $\mathbb{B}$-morphism $ f:B\longrightarrow GA$.
\item  A $G$-$\mathbf{costructured\ arrow\ with\ codomain}$ \index{$G$-costructured arrow} $B$ is a pair $(A,f)$ consisting of an $\mathbb{A}$-object $ A$ and a $\mathbb{B}$-morphism $ f:GA\longrightarrow B$.
\end{enumerate}
\end{definition}
\begin{definition}[$G$-universal arrow and $G$-couniversal arrow]\label{2.2}
$G$-universal arrow and $G$-couniversal arrow are defined as follows:
\begin{enumerate}
\item A $G$-structured arrow $(g,A)$ with domain $B$ is called $G$-$\mathbf{universal}$ \index{$G$-universal arrow} for $B$ provided that for each $G$-structured arrow $(g',A')$ with domain $B$, there exists a unique $\mathbb{A}$-morphism $\hat f:A \longrightarrow A'$ with $g'=G(\hat f)\circ g$. i.e., s.t. the triangle
\begin{center}
\begin{tikzpicture}[description/.style={fill=white,inner sep=2pt}] 
    \matrix (m) [matrix of math nodes, row sep=2.5em, column sep=2.5em]
    {B & &GA  \\
        & & GA' \\ }; 
    \path[->,font=\scriptsize] 
        (m-1-1) edge node[auto] {$g$} (m-1-3)
        (m-1-1) edge node[auto,swap] {$g'$} (m-2-3)
        (m-1-3) edge node[auto] {$G\hat f$} (m-2-3);
\end{tikzpicture}
\end{center}
commutes.
\\ We can also represent the above statement by the following diagram
\begin{center}    
\begin{tabular}{ l | r  } 
$\mathbb{B}$ & $\mathbb{A}$\\
\hline
{\begin{tikzpicture}[description/.style={fill=white,inner sep=2pt}] 
    \matrix (m) [matrix of math nodes, row sep=2.5em, column sep=2.5em]
    {B & &GA  \\
        & & GA' \\ }; 
    \path[->,font=\scriptsize] 
        (m-1-1) edge node[auto] {$g$} (m-1-3)
        (m-1-1) edge node[auto,swap] {$g'$} (m-2-3)
        (m-1-3) edge node[auto] {$G\hat f$} (m-2-3);
\end{tikzpicture}}  &  {\begin{tikzpicture}[description/.style={fill=white,inner sep=2pt}] 
    \matrix (m) [matrix of math nodes, row sep=2.5em, column sep=2.5em]
    {& &A  \\
        & &A' \\ }; 
    \path[->,font=\scriptsize]

        (m-1-3) edge node[auto] {$\hat f$} (m-2-3);
\end{tikzpicture}}\\
\end{tabular}
\end{center}
The diagram above indicates the fact that $g:B\longrightarrow GA$ is the $G$-universal arrow provided that for given $g':B\longrightarrow GA'$ there exist a unique $\mathbb{A}-morphisim$ $\hat f:A\longrightarrow A'$ s.t. the triangle commutes.
\item A $G$-costructured arrow $(A,g)$ with codomain $B$ is called $G$-$\mathbf{couniversal}$ \index{$G$-couniversal arrow} for $B$ provided that for each $G$-costructured arrow $(A',g')$ with codomain $B$, there exists a unique $\mathbb{A}$-morphism $\hat f:A' \longrightarrow A$ with $g'=g\circ G(\hat f)$. i.e., s.t. the triangle
\begin{center}
\begin{tikzpicture}[description/.style={fill=white,inner sep=2pt}] 
    \matrix (m) [matrix of math nodes, row sep=2.5em, column sep=2.5em]
    { GA&&B  \\
         GA'\\ }; 
    \path[->,font=\scriptsize] 
        (m-1-1) edge node[auto] {$g$} (m-1-3)
        (m-2-1) edge node[auto] {$G\hat f$} (m-1-1)
        (m-2-1) edge node[auto,swap] {$g'$} (m-1-3)
         ;
\end{tikzpicture}
\end{center}
commutes.
\\ We can also represent the above statement by the following diagram
\begin{center}    
\begin{tabular}{ l | r  } 
$\mathbb{B}$ & $\mathbb{A}$\\
\hline
{\begin{tikzpicture}[description/.style={fill=white,inner sep=2pt}] 
    \matrix (m) [matrix of math nodes, row sep=2.5em, column sep=2.5em]
    { GA&&B  \\
         GA'\\ }; 
    \path[->,font=\scriptsize] 
        (m-1-1) edge node[auto] {$g$} (m-1-3)
        (m-2-1) edge node[auto] {$G\hat f$} (m-1-1)
        (m-2-1) edge node[auto,swap] {$g'$} (m-1-3)
         ;
\end{tikzpicture}}  &  {\begin{tikzpicture}[description/.style={fill=white,inner sep=2pt}] 
    \matrix (m) [matrix of math nodes, row sep=2.5em, column sep=2.5em]
    {& &A  \\
        & &A' \\ }; 
    \path[->,font=\scriptsize]

        (m-2-3) edge node[auto] {$\hat f$} (m-1-3);
\end{tikzpicture}}\\
\end{tabular}
\end{center}
The diagram above indicates the fact that $g:GA\longrightarrow B$ is the $G$-couniversal arrow provided that for given $g':GA'\longrightarrow B'$ there exist a unique $\mathbb{A}-morphism$ $\hat f:A'\longrightarrow A$ s.t. the triangle commutes.
\end{enumerate}
\end{definition}
\begin{definition}[Left Adjoint and Right Adjoint]\label{2.3}
Left Adjoint and Right Adjoint are defined as follows.
\begin{enumerate}
\item A functor $G:\mathbb{A}\longrightarrow \mathbb{B}$ is said to be $\mathbf{left\ adjoint}$ \index{left adjoint} provided that  for every $\mathbb{B}$-object $B$, there exists a $G$-couniversal arrow with codomain $B$.
\\As a consequence, there exists a natural transformation $\eta:id_A\longrightarrow FG$ ($id_A$ is the identity morphism from $A$ to $A$), where  $F:\mathbb{B}\longrightarrow \mathbb{A}$ is a functor s.t. for given $f:A\longrightarrow FB$ there exists a unique $\mathbb{B}$-morphism $\hat f:GA\longrightarrow B$ s.t. the triangle
\begin{center}
\begin{tikzpicture}[description/.style={fill=white,inner sep=2pt}] 
    \matrix (m) [matrix of math nodes, row sep=2.5em, column sep=2.5em]
    { A&&FGA  \\
       &&FB\\ }; 
    \path[->,font=\scriptsize] 
        (m-1-1) edge node[auto] {$\eta_A$} (m-1-3)
        (m-1-1) edge node[auto,swap] {$f$} (m-2-3)
        (m-1-3) edge node[auto] {$F\hat f$} (m-2-3)
         ;
\end{tikzpicture}
\end{center}
commutes.
\\ This $\eta$ is called the unit \index{unit} of the adjunction.
\\ Hence, we have the diagram of unit \index{diagram of!unit} as follows:
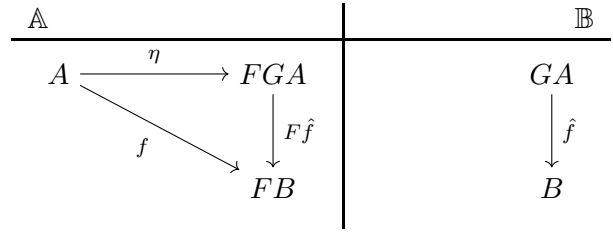
\begin{center}    
\begin{tabular}{ l | r  } 
$\mathbb{A}$ & $\mathbb{B}$\\
\hline
{\begin{tikzpicture}[description/.style={fill=white,inner sep=2pt}] 
    \matrix (m) [matrix of math nodes, row sep=2.5em, column sep=2.5em]
    {A & &FGA  \\
        & &FB \\ }; 
    \path[->,font=\scriptsize] 
        (m-1-1) edge node[auto] {$\eta$} (m-1-3)
        (m-1-1) edge node[auto,swap] {$f$} (m-2-3)
        (m-1-3) edge node[auto] {$F\hat f$} (m-2-3);
\end{tikzpicture}}  &  {\begin{tikzpicture}[description/.style={fill=white,inner sep=2pt}] 
    \matrix (m) [matrix of math nodes, row sep=2.5em, column sep=2.5em]
    {& &GA  \\
        & &B \\ }; 
    \path[->,font=\scriptsize]

        (m-1-3) edge node[auto] {$\hat f$} (m-2-3);
\end{tikzpicture}}\\
\end{tabular}
\end{center}
\item A functor $G:\mathbb{A}\longrightarrow \mathbb{B}$ is said to be $\mathbf{right\ adjoint}$ \index{right adjoint} provided that  for every $\mathbb{B}$-object $B$, there exists a $G$-universal arrow with domain $B$.
\\ From the definition above, it follows that there exists a natural transformation $\xi:FG\longrightarrow id_A$ ($id_A$ is the identity morphism from $A$ to $A$), where $F:\mathbb{B}\longrightarrow\mathbb{A}$ is a functor s.t. for given $f':FB\longrightarrow A$, there exists a unique $\mathbb{B}$-morphism $\hat f:B\longrightarrow GA$ s.t the triangle
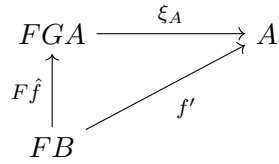
\begin{center}
\begin{tikzpicture}[description/.style={fill=white,inner sep=2pt}] 
    \matrix (m) [matrix of math nodes, row sep=2.5em, column sep=2.5em]
    { FGA&&A  \\
         FB\\ }; 
    \path[->,font=\scriptsize] 
        (m-1-1) edge node[auto] {$\xi_A$} (m-1-3)
        (m-2-1) edge node[auto] {$F\hat f$} (m-1-1)
        (m-2-1) edge node[auto,swap] {$f'$} (m-1-3)
         ;
\end{tikzpicture}
\end{center}
commutes.
\\ This $\xi$ is called the co-unit \index{co-unit} of the adjunction.
\\ Hence, we have the diagram of co-unit \index{diagram of!co-unit} as follows:
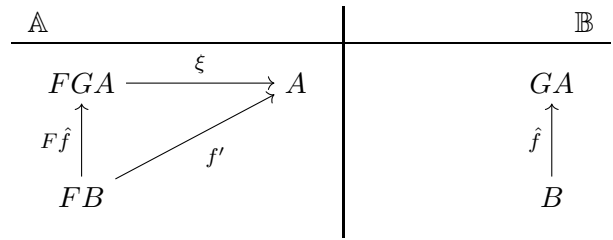
\begin{center}    
\begin{tabular}{ l | r  } 
$\mathbb{A}$ & $\mathbb{B}$\\
\hline
{\begin{tikzpicture}[description/.style={fill=white,inner sep=2pt}] 
    \matrix (m) [matrix of math nodes, row sep=2.5em, column sep=2.5em]
    { FGA&&A  \\
         FB\\ }; 
    \path[->,font=\scriptsize] 
        (m-1-1) edge node[auto] {$\xi$} (m-1-3)
        (m-2-1) edge node[auto] {$F\hat f$} (m-1-1)
        (m-2-1) edge node[auto,swap] {$f'$} (m-1-3)
         ;
\end{tikzpicture}}  &  {\begin{tikzpicture}[description/.style={fill=white,inner sep=2pt}] 
    \matrix (m) [matrix of math nodes, row sep=2.5em, column sep=2.5em]
    {& &GA  \\
        & &B \\ }; 
    \path[->,font=\scriptsize]

        (m-2-3) edge node[auto] {$\hat f$} (m-1-3);
\end{tikzpicture}}\\
\end{tabular}
\end{center}
\end{enumerate}
\end{definition} 
\begin{definition}[Heyting algebra]\label{gf}
An algebra $(A, \vee, \wedge, \rightarrow, \mathbf{1}, \mathbf{0})$ with three binary and two nullary operations is said to be $\textbf{Heyting algebra}$ if  $(A, \vee, \wedge, \mathbf{1}, \mathbf{0})$ is a bounded distributive lattice and $\rightarrow$ is a  binary operation which is adjoint to
 $\wedge$.
\end{definition}
\begin{definition}[G\"{o}del algebra]\label{ga}
A Heyting algebra $A$ satisfying the prelinearity property viz. $(a\rightarrow b)\vee (b\rightarrow a)=\mathbf{1}$, for any $a,b\in A$ is said to be a \textbf{G\"{o}del algebra}.
\end{definition}
\begin{definition}[Heyting homomorphism]\label{hhm}
Let $A$, $B$ be two Heyting algebras. A map $f:A\longrightarrow B$ is said to be $\textbf{Heyting homomorphism}$ if the following conditions hold:\\
(i) $f(a_1\wedge a_2)=f(a_1)\wedge f(a_2)$;\\
(ii) $f(a_1\vee a_2)=f(a_1)\vee f(a_2)$;\\
(iii) $f(a_1\rightarrow a_2)=f(a_1)\rightarrow f(a_2)$;\\
(iv) $f(\mathbf{0})=\mathbf{0}$.
\end{definition}
\textbf{\underline{Note:}} The set of bounded distributive lattice homomorphisms from a Heyting algebra $A$ to the Heyting algebra $(\{0,1\},\vee,\wedge,\rightarrow,1,0)$ will be denoted by $Hom(A,\{0,1\})$ in this paper.

Let us consider the example:
\begin{figure}[h!]
\begin{center}
\begin{tikzpicture}[scale=.7]
\node (one) at (0,2) {$1$};
\node (a) at (0,0) {$a$};
\node (zero) at (0,-2) {$0$};
\node (1) at (4,1) {$1$};
\node (0) at (4,-1) {$0$};
\draw (zero) -- (a) -- (one);
\draw (0) -- (1);
\draw [-latex,blue] (one) -- (1);
\draw [-latex,blue] (a) -- (1);
\draw [-latex,blue] (zero) -- (0);
\node (onen) at (8,2) {$1$};
\node (an) at (8,0) {$a$};
\node (zeron) at (8,-2) {$0$};
\node (1n) at (12,1) {$1$};
\node (0n) at (12,-1) {$0$};
\draw (zeron) -- (an) -- (onen);
\draw (0n) -- (1n);
\draw [-latex,blue] (onen) -- (1n);
\draw [-latex,blue] (an) -- (0n);
\draw [-latex,blue] (zeron) -- (0n);
\end{tikzpicture}
\caption{}
\end{center}
\end{figure}

We have two lattice homomorphisms 
$$h_1(1)=h_1(a)=1\ , h_1(0)=0 \ \text{and}\ h_2(1)=1,\ h_2(a)=h_2(0)=0,\ h_2\leq h_1.$$
Here $h_2\leq h_1$ iff $h_2(a)\leq h_1(a)$ for any $a\in A$.  $h_2$ is not Heyting homomorphism. $h_1$ is the only Heyting homomorphism, which is maximal. But there exists two prime filters in the Heyting algebra (and in the lattice as well):
$$F_1=\{1,a\}\ \text{and}\ F_2=\{1\},\ F_2\subseteq F_1.$$
It is well known Priestley duality between bounded distributive lattices and Priestley spaces $(X,R)$ \cite{BD,HA}. Priestley space is Heyting space (or Esakia space) \cite{MA} if and only if 
$$(*) R^{-1}(U)\ \text{is open for every open set}\ U.$$ So in the construction of Heyting space (or Esakia space) we use Priestley space with the condition $(*)$.

Notice, the restricted Priestley duality for Heyting algebras states that a bounded distributive lattice $A$ is a Heuting algebra if and only if the Priestley dual of $A$ is a Heyting space and a $\{0,1\}$-lattice homomorphism $h$ between Heyting algebras preserves the implication $\rightarrow$ if and only if the Priestley dual of $h$ is a Heyting morphism.
\begin{definition}[{$\mathbf{HA}$}]\label{Grfrm} 
Heyting algebras together with Heyting homomorphisms form a category, which is well known as a category of Heyting algebras and denoted by $\mathbf{HA}$.
\end{definition}
\begin{definition}[{$\mathbf{GA}$}]\label{GA} 
G\"{o}del algebras together with corresponding Heyting homomorphisms form a category, which is well known as a category of G\"{o}del algebras and denoted by $\mathbf{GA}$.
\end{definition}
\begin{definition}[Esakia Space]\label{ES}
An ordered topological space $(X,\leq,\tau)$ is called an $\textbf{Esakia space}$ if 
\begin{itemize}
\item $(X,\tau)$ is compact;
\item for any $x, y\in X$ with $x\nleq y$ there exists a clopen up-set $U\subseteq X$ with $x\in U$, $y\notin U$;
\item for any clopen set $U$, the down-set $\downarrow U$ is also clopen.
\end{itemize}
\end{definition}
Note that an ordered topological space $(X,\leq,\tau)$ together with the first two conditions of Definition \ref{ES} is known as Priestley space.
\begin{definition}[Esakia morphism]
Let $(X,\leq,\tau)$ and $(Y,\leq,\tau')$ be Esakia spaces. Then a map $f:X\longrightarrow Y$ is called an $\textbf{Esakia morphism}$ if $f$ is a continuous bounded morphism (p-morphism), i.e., if for each $x\in X$ and $y\in Y$, if $f(x)\leq y$, then there exists $z\in X$ such that $x\leq z$ and $f(z)=y$.
\end{definition}
\begin{definition}
Esakia spaces together with Esakia morphisms forms a category of Esakia spaces and denoted by $\mathbf{ESA}$.
\end{definition}
\begin{theorem}\label{ESD} \cite{LE}
$\mathbf{HA}$ is dually equivalent with $\mathbf{ESA}$.
\end{theorem}
\section{Categories: I Top, I TopSys, HA and their interrelationships}
Suppose that we have the algebras $A$ and $B$, and two homomorphisms $h_1$, $h_2$ from $A$ to $B$. Then we can define the ordering $R$ on the set of all homomorphisms from $A$ to $B$:
$$h_1Rh_2\ \text{iff}\ h_1(a)\leq h_2(a)\ \text{for all}\ a\in A.$$
So, (Hom(A,\{0,1\}),R) is a poset, where $A$ is a Heyting algebra and $Hom(A,\{0,1\})$ is the set of all \textbf{bounded distributive lattice} homomorphisms from $A$ to $(\{0,1\},\vee,\wedge,\rightarrow,1,0)$.
\begin{definition}[I-topological system]\label{gftsy}
An \textbf{I-topological system} is a triple $(X,\models,A)$ consisting of a nonempty set $X$, a Heyting algebra $A$ and a relation $\models$ from $X$ to $A$ such that
\begin{enumerate}
\item $x\models \mathbf{0}$ for no $x\in X$;
\item $x\models a\wedge b$ iff $x\models a$ and $x\models b$;
\item $x\models a\vee b$ iff $x\models a$ or $x\models b$;

\item $x\models a\rightarrow b$ iff for all $y\in X$ such that $p^*(x)Rp^*(y)$, $y\not\models a$ or $y\models b$, where $p^*:X\to Hom(A,\{0,1\})$ such that $p^*(x)(a)=1\ \text{iff}\ x\models a$.
\end{enumerate} 
\end{definition}
From Definition \ref{gftsy} it is easy to deduce that:
$$x\models\neg a\ \text{iff}\ \text{for all}\ y\in X\ \text{such that}\ p^*(x)Rp^*(y), y\not\models a.$$
Now, let us show that $x\not\models a\vee\neg a$, for some $x\in X$. Let $X=\{x,y\}$ and $A=(\{0,a,1\},\vee,\wedge,\rightarrow,1,0)$, where $0\leq a\leq 1$. Then we have two bounded distributive lattice homomorphisms $p'(x)(=h_2)$ and $p'(y)(=h_1)$ ($h_1$ and $h_2$ are represented in Figure 1) and $p'(x)\leq p'(y)$. Let us consider $$x\models a\ \text{iff}\ p'(x)(a)=1.$$ Then clearly $y\models a$ and $x\not\models a$. So it can be derived that $x\not\models \neg a$. Hence $y\models a\vee\neg a$ but $x\not\models a\vee\neg a$. Consequently we may conclude that for this choice of $x\in X$, $x\not\models a\vee\neg a$.
\begin{proposition}\label{prop1}
$x\models \mathbf{1}$ for any $x\in X$.
\end{proposition}
\begin{proof}
$x\models \mathbf{1}$ iff $x\models a\rightarrow a$ iff for all $y\in X$ such that $p^*(x)Rp^*(y)$, $y\not\models a$ or $y\models a$. As for any $x\in X$ and $a\in A$ either $x\models a$ or $x\not\models a$ holds, $x\models \mathbf{1}$ for any $x\in X$.
\end{proof}
\begin{definition}[Heyting algebraic I-topological system]\label{HITopSys}
An I-topological system $(X,\models,A)$ is said to be \textbf{Heyting algebraic} if the map $p^*:X\longrightarrow Hom(A,\{0,1\})$ defined by, $p^*(x)(a)=1$ iff $x\models a$ for $x\in X$ and $a\in A$, is a bijective mapping.
\end{definition}
\begin{definition}
An I-topological system $(X,\models, A)$ is said to be $\mathbf{T_0}$ iff (if $x_1\neq x_2$ then there exist some $a\in A$ such that $x_1\models a$ but $x_2\not\models a$).
\end{definition}
\begin{proposition}
Any Heyting algebraic I-topological system is $T_0$.
\end{proposition}
\begin{proof}
For Heyting algebraic I-topological system $(X,\models,A)$, the map $p^*:X\longrightarrow Hom(A,\{0,1\})$ is bijective and consequently injective. Hence if $x_1\neq x_2$ then $p^*(x_1)\neq p^*(x_2)$ and hence there exist $a\in A$ such that $p^*(x_1)(a)\neq p^*(x_2)(a)$. So as per the definition of $p^*$ it is clear that the system is $T_0$.
\end{proof}

\begin{definition}[G\"{o}del algebraic I-topological system]\label{GITopSys}
A \textbf{G\"{o}del algebraic I-topological system} is a triple $(X,\models,A)$ consisting of a non empty set $X$, a G\"{o}del algebra $A$ and a binary relation $\models$ from $X$ to $A$ such  that 
\begin{enumerate}
\item $x\models \mathbf{0}$ for no $x\in X$;
\item $x\models a\wedge b$ iff $x\models a$ and $x\models b$;
\item $x\models a\vee b$ iff $x\models a$ or $x\models b$;
\item $x\models a\rightarrow b$ iff for all $y\in X$ such that $p^*(x)Rp^*(y)$, $y\not\models a$ or $y\models b$, where $p^*:X\to Hom(A,\{0,1\})$ such that $p^*(x)(a)=1\ \text{iff}\ x\models a$;
\item the map $p^*:X\longrightarrow Hom(A,\{0,1\})$ defined by, $p^*(x)(a)=1$ iff $x\models a$ for $x\in X$ and $a\in A$, is a bijective mapping.
\end{enumerate} 
\end{definition}
\subsection{Kripke model for intuitionistic logic and I-topological system}
In this subsection we will indicate the connection of the notion of I-topological system with Kripke model for intuitionistic logic \cite{AM}.
\begin{definition}
A Kripke frame $\mathscr{F}$ is a pair $(W,\mathscr{R})$ consisting of a nonempty set of worlds (or points), $W$, and a partial order relation $\mathscr{R}$ on $W$ ($\mathscr{R}\subseteq W\times W$). 
\end{definition}
\begin{definition}
A Kripke model $\mathscr{M}$ is a pair $(\mathscr{F},v)$ consisting of a Kripke frame $\mathscr{F}$ and a valuation map $v:W\times \mathbf{V}\to \{0,1\}$, where $\mathbf{V}$ is the set of propositional variables such that:
\begin{enumerate}
    \item for all $w\in W$ and for all propositional variables $p\in \mathbf{V}$, if $v(w,p)=1$ and $w\mathscr{R} u$ then $v(u,p)=1$;
    \item $v(w,\bot)=0$ for all $w\in W$.
\end{enumerate}
\end{definition}
\begin{definition}
Let $\mathscr{M}$ be a Kripke model for intuitionistic logic and $w$ be a world in the frame $\mathscr{F}$. By induction on the construction of a formula $a$ we define a relation $(\mathscr{M},w)\Vdash a$, which is read as ``$a$ is true at $w$ in $\mathscr{M}$":
\begin{itemize}
    \item $\mathscr{M},w\Vdash p\ \text{iff}\ v(w,p)=1$;
    \item $\mathscr{M},w\Vdash a\wedge b\ \text{iff}\ \mathscr{M},w\Vdash a\ \text{and}\ \mathscr{M},w\Vdash b$;
    \item $\mathscr{M},w\Vdash a\vee b\ \text{iff}\ \mathscr{M},w\Vdash a\ \text{or}\ \mathscr{M},w\Vdash b$;
    \item $\mathscr{M},w\Vdash \neg a\ \text{iff}\ \forall u\geq w,\ \mathscr{M},u\not\Vdash a$;
    \item $\mathscr{M},w\Vdash a\rightarrow b\ \text{iff}\ \forall w\mathscr{R} u,\ \text{if}\ \mathscr{M},u\Vdash a \ \text{then} \mathscr{M},u\Vdash b$;
    \item $\mathscr{M},w\not\Vdash\bot$.
\end{itemize}
\end{definition}
Let $(X,\models , A)$ be an I-topological system. Then consider the relation $\mathscr{R}$ on $X$ such that 
$$x\mathscr{R}y\ \text{iff}\ p^*(x)Rp^*(y),\ \text{where}\ p^*(x)(a)=1\ \text{iff}\ x\models a.$$
It may be noticed that $(X,\mathscr{R})$ is a partially ordered set. Hence $(X,\mathscr{R})$ is a Kripke frame.

Moreover if we consider $v:X\times A\to \{0,1\}$ such that $v(x,a)=1\ \text{iff}\ x\models a$ then the following holds.
\begin{enumerate}
    \item For all $x\in X$ and for all $a\in A$ let us assume that $v(x,a)=1$ and $x\mathscr{R}y$. Then we have $x\models a$ and $p^*(x)Rp^*(y)$ i.e. $p^*(x)(a)\leq p^*(y)(a)$. As $x\models a$, $p^*(x)(a)=1=p^*(y)(a)$. Hence $y\models a$. Therefore for all $x\in X$ and for all $a\in A$, $v(x,a)=1$  and $x\mathscr{R}y$ implies $v(y,a)=1$.
    \item We know $v(x,\mathbf{0})=1\ \text{iff}\ x\models \mathbf{0}$. But $x\models \mathbf{0}$ for no $x\in X$. Hence for all $x\in X$, $v(x,\mathbf{0})=0$.
\end{enumerate}
Consequently $(X,\mathscr{R},v)$ is a Kripke model.

Now let us define $x\Vdash a$ iff $x\models a$. Then,
\begin{itemize}
    \item $x\Vdash a$ iff $x\models a$ iff $v(x,a)=1$;
    \item $x\Vdash a\wedge b$ iff $x\models a\wedge b$ iff $x\models a$ and $x\models b$ iff $x\Vdash a$ and $x\Vdash b$; 
    \item $x\Vdash a\vee b$ iff $x\models a\vee b$ iff $x\models a$ or $x\models b$ iff $x\Vdash a$ or $x\Vdash b$; 
    \item Let $x\Vdash a\rightarrow b$. Then,
    \begin{align*}
    x\Vdash a\rightarrow b\  & \text{iff}\ x\models a\rightarrow b \\
    & \text{iff for all}\ y\in X\  \text{such that}\ p^*(x)Rp^*(y), y\not\models a\ \text{or}\ y\models b \\
    & \text{iff} \text{ for all}\ y\in X \ \text{and}\ p^*(x)Rp^*(y),\ y\not\Vdash a\ \text{or}\ y\Vdash b \\
    & \text{iff for all}\ y\in X\ \text{and}\ x\mathscr{R}y,\ \text{if}\ y\Vdash a\ \text{then}\ y\Vdash b;
    \end{align*}
    \item As $x\models \mathbf{0}$ for no $x\in X$, $x\not\Vdash \mathbf{0}$.
\end{itemize}
Summarizing all sayings above  we can deduce the following 
Theorem.
\begin{theorem}
Let $(X,\models,A)$ be an I-topological system. Then $(X,\mathscr{R},v)$, defined as above, is an intuitionistic Kripke model.
\end{theorem}
\subsection{Categories}

\begin{definition}[{$\mathbf{I-TopSys}$}]\label{Grftsy} 
The category $\mathbf{I-TopSys}$\index{category!Graded Fuzzy TopSys} is defined thus.
\begin{itemize}
\item The objects are I-topological systems $(X,\models,A)$, $(Y,\models,B)$ etc. (c.f. Definition \ref{gftsy}).
\item The morphisms are pair of maps satisfying the following continuity properties: If $(f_1,f_2):(X,\models ,A)\longrightarrow (Y,\models',B)$ then\\
(i) $f_1:X\longrightarrow Y$ is a set map;\\
(ii) $f_2:B\longrightarrow A$ is a Heyting homomorphism;\\
(iii) $x\models f_2(b)$ iff $f_1(x)\models' b$.
\item The identity on $(X,\models,A)$ is the pair $(id_X,id_A)$, where $id_X$ is the identity map on $X$ and $id_A$ is the identity Heyting homomorphism. That this is an $\mathbf{I-TopSys}$ morphism\index{Graded Fuzzy TopSys!morphism} can be proved.
\item If $(f_1,f_2):(X,\models,A)\longrightarrow (Y,\models',B)$ and 
$(g_1,g_2):(Y,\models',B)\longrightarrow (Z,\models'',C)$ are morphisms in $\mathbf{I-TopSys}$, their 
composition $(g_1,g_2)\circ (f_1,f_2)=(g_1\circ f_1,f_2\circ g_2)$ is the pair of composition of functions between two sets and composition of Heyting homomorphisms between two Heyting algebras. It can be verified that 
$(g_1,g_2)\circ (f_1,f_2)$ is a morphism in $\mathbf{I-TopSys}$.
\end{itemize}
\end{definition}
\begin{definition}[{$\mathbf{HI-TopSys}$}]\label{hits}
Heyting algebraic I-topological systems (c.f. Definition \ref{HITopSys}) together with corresponding $\mathbf{I-TopSys}$ morphisms form a category and called $\mathbf{HI-TopSys}$.
\end{definition}
\begin{definition}[{$\mathbf{GI-TopSys}$}]\label{gits}
G\"{o}del algebraic I-topological systems (c.f. Definition \ref{GITopSys}) together with corresponding $\mathbf{I-TopSys}$ morphisms form a category and called $\mathbf{GI-TopSys}$.
\end{definition}

\subsection{Functors }
Let us construct suitable functors among the above mentioned categories as follows to establish their interrelations. 

\begin{definition}
$H$ is a functor from $\mathbf{HI-TopSys}$ to $\mathbf{HA^{op}}$ defined as follows:\\
$H$ acts on an object $(X,\models, A)$ as $H((X,\models , A))=A$ and on a morphism $(f_1,f_2)$ as $H((f_1,f_2))=f_2$.
\end{definition}
It is easy to verify that $H$ is indeed a functor.
\begin{definition}
$\mathscr{G}$ is a functor from $\mathbf{GI-TopSys}$ to $\mathbf{GA^{op}}$ defined as follows:\\
$\mathscr{G}$ acts on an object $(X,\models, A)$ as $\mathscr{G}((X,\models , A))=A$ and on a morphism $(f_1,f_2)$ as $\mathscr{G}((f_1,f_2))=f_2$.
\end{definition}
It is easy to verify that $\mathscr{G}$ is indeed a functor.

\begin{lemma}\label{lemma1}
$(Hom(A,\{0,1\}),\models^*,A)$, where $A$ is a Heyting algebra and $v\models^* a$ iff $v(a)=1$, is an $I$-topological system.
\end{lemma}
\begin{proof}
Let us proceed in the following way.

(i) $v\models^*\mathbf{0}$ iff $v(\mathbf{0})=1$, but as $v$ is a bounded distributive lattice homomorphism so, $v(\mathbf{0})=0$. Hence $v\models^*\mathbf{0}$ for no $v\in Hom(A,\{0,1\})$.

(ii) $v\models^* a\wedge b$ iff $v(a\wedge b)=1$ iff $v(a)\wedge v(b)=1$ iff $v(a)=1$ and $v(b)=1$ iff $v\models^* a$ and $v\models^* b$.

(iii) $v\models^* a\vee b$ iff $v(a\vee b)=1$ iff $v(a)\vee v(b)=1$ iff $v(a)=1$ or $v(b)=1$ iff $v\models^* a$ or $v\models^* b$.

(iv) Let us assume that $v\models^* a\rightarrow b$. We have $v\models^* a\rightarrow b$ iff $v(a\rightarrow b)=1$. Now for any $v'\in Hom(A,\{0,1\})$ such that $v\leq v'$, we have $v'(a\rightarrow b)=1$. So $v'(a)\rightarrow v'(b)=1$. Hence $v'(a)=0$ or $v'(b)=1$. Consequently $v'\not\models^* a$ or $v'\models^* b$ for any $v'\in Hom(A,\{0,1\})$ such that $vRv'$.

Let for all $v'\in Hom(A,\{0,1\})$ such that $vRv'$, $v'\not\models^*a$ or $v'\models^* b$, i.e., $v'(a)=0$ or $v'(b)=1$. In particular we have $v(a)=0$ or $v(b)=1$. We need to show that $v(a\rightarrow b)=1$ i.e., $v\models^* a\rightarrow b$. For any Heyting algebra $A$ and $a,b\in A$ it is known that $b\leq a\rightarrow b$ and so $v(b)\leq v(a\rightarrow b)$. Hence for $v(b)=1$, $v(a\rightarrow b)=1$. Now when $v(b)=0$, if possible let us assume that $v(a\rightarrow b)=0$. Now $v^{-1}(0)$ is an ideal so $v(a\rightarrow b)=0$ and $v(b)=0$ implies $v(a)=0$. In this case $v(a)\rightarrow v(b)=1$ but it is possible to choose $w\in Hom(A,\{0,1\})$ such that $w(a)=1$ and $w(b)=0$. For this choice of $w$ it is clear that $vRw$ but $w\models^* a$ and $w\not\models^* b$, which contradicts our assumption. Hence $v(a\rightarrow b)=1$ for this case.

Hence we can conclude that $v\models^* a\rightarrow b$ iff for all $v'\in Hom(A,\{0,1\})$ such that $vRv'$, $v'\not\models^*a$ or $v'\models^* b$.
\end{proof}
\begin{corollary}\label{cor1}
$(Hom(A,\{0,1\}),\models^*,A)$, where $A$ is a G$\ddot{o}$del algebra and $v\models^* a$ iff $v(a)=1$, is an $I$-topological system.
\end{corollary}
\begin{lemma}\label{lemma2}
For any Heyting algebra $A$, $(Hom(A,\{0,1\}),\models^*,A)$ have the following properties.\\
(i) if for any $a,b\in A$, $(v\models^* a$ iff $v\models^*b$ for any $v\in Hom(A,\{0,1\}))$ then $a=b$.\\
(ii) if $v_1\neq v_2$ then there exist $a\in A$ such that $v_1\models^* a$ but $v_2\not\models^*a$.\\
(iii) $p^*:Hom(A,\{0,1\})\longrightarrow Hom(A,\{0,1\})$ defined by, $p^*(v)(a)=1$ iff $v\models^*a$ is a bijection.
\end{lemma}
\begin{proof}
(i) Let for any $a,b\in A$ and $v\in Hom(A,\{0,1\})$, $v\models^* a$ iff $v\models^*b$. So, $v(a)=1$ iff $v(b)=1$ for any $v\in Hom(A,\{0,1\})$. Hence $a=b$ can be concluded.
\\
Properties (ii) and (iii) can be verified by routine check.
\end{proof}
\begin{corollary}\label{cor2}
For any G\"{o}del algebra $A$, $(Hom(A,\{0,1\}),\models^*,A)$ have the following properties.\\
(i) if for any $a,b\in A$, $(v\models^* a$ iff $v\models^*b$ for any $v\in Hom(A,\{0,1\}))$ then $a=b$.\\
(ii) if $v_1\neq v_2$ then there exist $a\in A$ such that $v_1\models^* a$ but $v_2\not\models^*a$.
\end{corollary}
\begin{lemma}\label{lemma3}
If $f:B\longrightarrow A$ is a Heyting homomorphism then $(\_\circ f,f):(Hom(A,\{0,1\}),\models^*,A)\longrightarrow (Hom(B,\{0,1\}),\models^*,B)$ is continuous.
\end{lemma}
\begin{proof}
We have $v\models^* f(b)$ iff $v(f(b))=1$ iff $v\circ f(b)=1$ iff $(\_\circ f(v))(b)=1$ iff $\_\circ f(v)\models^* b$.
\end{proof}
\begin{definition}
$S$ is a functor from $\mathbf{HA^{op}}$ to $\mathbf{I-TopSys}$ defined as follows. $S$ acts on an object $A$ as $S(A)=(Hom(A,\{0,1\}),\models^*,A)$ and on a morphism $f$ as $S(f)=(\_\circ f,f)$ (it is indeed a functor follows from Lemma \ref{lemma1} and Lemma \ref{lemma3}).
\end{definition}
\begin{proposition}\label{proposition1}
$S$ is a functor from $\mathbf{HA^{op}}$ to $\mathbf{HI-TopSys}$.
\end{proposition}
\begin{proof}
Follows from Lemma \ref{lemma1}, Lemma \ref{lemma2} and Lemma \ref{lemma3}.
\end{proof}

\begin{definition}
$\mathscr{S}$ is a functor from $\mathbf{GA^{op}}$ to $\mathbf{GI-TopSys}$ defined as follows. $\mathscr{S}$ acts on an object $A$ as $\mathscr{S}(A)=(Hom(A,\{0,1\}),\models^*,A)$ and on a morphism $f$ as $\mathscr{S}(f)=(\_\circ f,f)$ (it is indeed a functor follows from Corollary \ref{cor1}, Corollary \ref{cor2} and Lemma \ref{lemma3}).
\end{definition}
\begin{theorem}\label{theorem1}
$\mathbf{HI-TopSys}$ is dually equivalent to $\mathbf{HA}$.
\end{theorem}
\begin{proof}
First we will prove that $H$ is the left adjoint to the functor $S$ by presenting the unit of the adjunction.

Recall that  $S(A)=(Hom(A,\{0,1\}),\models^*,A)$ where $v\models^* a$ iff $v(a)=1$ and $H((X,\models,A))=A$.

Hence $S(H((X,\models, A)))=(Hom(A,\{0,1\}),\models^*, A)$.
\begin{center}
\begin{tabular}{ l | r  } 
 $\mathbf{HI-TopSys}$ & $\mathbf{HA^{op}}$ \\
\hline
 {\begin{tikzpicture}[description/.style={fill=white,inner sep=2pt}] 
    \matrix (m) [matrix of math nodes, row sep=2.5em, column sep=2.5em]
    {(X,\models ,A) & &S(H((X, \models ,A)))  \\
        & & S(B) \\ }; 
    \path[->,font=\scriptsize] 
        (m-1-1) edge node[auto] {$\eta$} (m-1-3)
        (m-1-1) edge node[auto,swap] {$f(\equiv(f_1,f_2))$} (m-2-3)
        (m-1-3) edge node[auto] {$S\hat{f}$} (m-2-3);
\end{tikzpicture}} &  {\begin{tikzpicture}[description/.style={fill=white,inner sep=2pt}] 
    \matrix (m) [matrix of math nodes, row sep=2.5em, column sep=2.5em]
    {& &H((X,\models ,A))  \\
        & &B \\ }; 
    \path[->,font=\scriptsize]

        (m-1-3) edge node[auto] {$\hat{f}(\equiv f_2)$} (m-2-3);
\end{tikzpicture}}\\
\end{tabular}
\end{center}
Then unit is defined by $\eta =(p^*,id_A)$.\\
 \begin{tikzpicture}[description/.style={fill=white,inner sep=2pt}] 
    \matrix (m) [matrix of math nodes, row sep=2.5em, column sep=2.5em]
    {i.e.\ (X,\models ,A) & &S(H((X, \models ,A)))  \\
        }; 
    \path[->,font=\scriptsize] 
        (m-1-1) edge node[auto] {$\eta$} (m-1-3)
        (m-1-1) edge node[auto,swap] {$(p^*,id_A)$} (m-1-3)
        ;
\end{tikzpicture}

where $\\p^*:X\longrightarrow Hom(A,\{0,1\})$,

$x\longmapsto p_x:A\longrightarrow \{0,1\}$
such that $p_x(a)=1$ iff $x\models a$.

If possible let $p_x(\mathbf{0})=1$. Then we have $x\models 0$, which is a contradiction as $x\models 0$ for no $x\in X$. Hence $p_x(\mathbf{0})=0$. Also we have $p_x(a_1\wedge a_2)=1$ iff $x\models a_1\wedge a_2$ iff $x\models a_1$ and $x\models a_2$ iff $p_x(a_1)=1$ and $p_x(a_2)=1$ iff $p_x(a_1)\wedge p_x(a_2)=1$. Similarly it can be shown that $p_x(a_1\vee a_2)=p_x(a_1)\vee p_x(a_2)$ and $p_x(a_1\rightarrow a_2)=p_x(a_1)\rightarrow p_x(a_2)$. Hence for each $x\in X,$ $p_x:A\longrightarrow \{0,1\}$ is a Heyting homomorphism.

It may be observed that $x\models id_A(a)$ iff $x\models a$ iff $p_x(a)=1$ iff $(p^*(x))(a)=1$ iff $p^*(x)\models^* a$. Consequently we can conclude that  $(p^*,id_A):(X,\models,A)\longrightarrow S(H((X,\models,A))) $ is a continuous map of Heyting algebraic I-topological system.

Let us define $\hat{f}$ as follows:
$(f_1,f_2):(X,\models ,A)\longrightarrow (Hom(B,\{0,1\}),\models_*,B)$
\\then $\hat{f}=f_2$.
Recall that $S(\hat{f})=(\_\circ f_2,f_2)$.

It suffices to show that the triangle on the left commute, i.e., $(f_1,f_2)=S(\hat{f})\circ \eta$. Now, $S(\hat{f})\circ \eta=(\_\circ f_2,f_2)\circ (p^*,id_A)=((\_\circ f_2)\circ p^*,id_A\circ f_2)=((\_\circ f_2)\circ p^*, f_2)$. For any $x\in X,$ $f_1(x)=(\_\circ f_2)\circ p^*(x)=(\_\circ f_2)\circ p_x=p_x\circ f_2$. Consequently, for all $b\in B$, $f_1(x)(b)=1$ iff $f_1(x)\models^* b$ iff $x\models f_2(b)$ iff $p_x(f_2(b))=1$ iff $(p_x\circ f_2)(b)=1$ iff $((\_\circ f_2)\circ p_x)(b)=1$ iff $((\_\circ f_2)\circ p^*)(x)(b)=1$. Therefore $f_1=(\_\circ f_2)\circ p^*$. Hence $\eta(\equiv(p^*,id_A)):(X,\models ,A)\longrightarrow S(H((X,\models ,A)))$ is the unit, consequently $H$ is the left adjoint to the functor $S$.

Diagram of the co-unit of the above adjunction is as follows.

\begin{center}
\begin{tabular}{ l | r  } 
 $\mathbf{HA^{op}}$ & $\mathbf{HI-TopSys}$ \\
\hline
 {\begin{tikzpicture}[description/.style={fill=white,inner sep=2pt}] 
    \matrix (m) [matrix of math nodes, row sep=2.5em, column sep=2.5em]
    {H(S(A)) & & A  \\
         H((Y,\models ,B)) \\ }; 
    \path[->,font=\scriptsize] 
        (m-1-1) edge node[auto] {$\xi(\equiv id_A)$} (m-1-3)
        (m-2-1) edge node[auto] {$f$} (m-1-1)
        (m-2-1) edge node[auto,swap] {$H\hat{f}$} (m-1-3);
\end{tikzpicture}} &  {\begin{tikzpicture}[description/.style={fill=white,inner sep=2pt}] 
    \matrix (m) [matrix of math nodes, row sep=2.5em, column sep=2.5em]
    {& &S(A) \\
        & &(Y,\models,B) \\ }; 
    \path[->,font=\scriptsize]

        (m-2-3) edge node[auto,swap] {$\hat{f}(\equiv\_\circ f)$} (m-1-3);
\end{tikzpicture}}\\
\end{tabular}
\end{center}
From the construction it can be easily shown that $\xi$ and $\eta$ are natural isomorphism and hence the theorem holds.
\end{proof}
\begin{corollary}\label{CA}
There exist adjoint functors between $\mathbf{HA^{op}}$ and $\mathbf{I-TopSys}$.
\end{corollary}

\begin{theorem}\label{ED}
There exist adjoint functors between  $\mathbf{ESA}$ and $\mathbf{HA^{op}}$.
\end{theorem}
\begin{proof}
Follows from Theorem \ref{ESD}.
\end{proof}
\begin{theorem}
There exist adjoint functors between $\mathbf{ESA}$ and $\mathbf{I-Top Sys}$.
\end{theorem}
\begin{proof}
Follows from Corollary \ref{CA} and Theorem \ref{ED}.
\end{proof}
\begin{theorem}
Category $\mathbf{HI-TopSys}$ is equivalent to $\mathbf{ESA}$.
\end{theorem}
\begin{proof}
Follows from Theorem \ref{theorem1} and Theorem \ref{ESD}.
\end{proof}
\begin{theorem}\label{GIGA}
$\mathbf{GI-TopSys}$ is dually equivalent to $\mathbf{GA}$.
\end{theorem}
\begin{proof}
First we will prove that $\mathscr{G}$ is the left adjoint to the functor $\mathscr{S}$ by presenting the unit of the adjunction.

Recall that  $\mathscr{S}(A)=(Hom(A,\{0,1\}),\models^*,A)$ where $v\models^* a$ iff $v(a)=1$ and $\mathscr{G}((X,\models,A))=A$.

Hence $\mathscr{S}(\mathscr{G}((X,\models, A)))=(Hom(A,\{0,1\}),\models^*, A)$.
\begin{center}
\begin{tabular}{ l | r  } 
 $\mathbf{GI-TopSys}$ & $\mathbf{GA^{op}}$ \\
\hline
 {\begin{tikzpicture}[description/.style={fill=white,inner sep=2pt}] 
    \matrix (m) [matrix of math nodes, row sep=2.5em, column sep=2.5em]
    {(X,\models ,A) & &\mathscr{S}(\mathscr{G}((X, \models ,A)))  \\
        & & \mathscr{S}(B) \\ }; 
    \path[->,font=\scriptsize] 
        (m-1-1) edge node[auto] {$\eta$} (m-1-3)
        (m-1-1) edge node[auto,swap] {$f(\equiv(f_1,f_2))$} (m-2-3)
        (m-1-3) edge node[auto] {$\mathscr{S}\hat{f}$} (m-2-3);
\end{tikzpicture}} &  {\begin{tikzpicture}[description/.style={fill=white,inner sep=2pt}] 
    \matrix (m) [matrix of math nodes, row sep=2.5em, column sep=2.5em]
    {& &\mathscr{G}((X,\models ,A))  \\
        & &B \\ }; 
    \path[->,font=\scriptsize]

        (m-1-3) edge node[auto] {$\hat{f}(\equiv f_2)$} (m-2-3);
\end{tikzpicture}}\\
\end{tabular}
\end{center}
Then unit is defined by $\eta =(p^*,id_A)$.\\
 \begin{tikzpicture}[description/.style={fill=white,inner sep=2pt}] 
    \matrix (m) [matrix of math nodes, row sep=2.5em, column sep=2.5em]
    {i.e.\ (X,\models ,A) & &\mathscr{S}(\mathscr{G}((X, \models ,A)))  \\
        }; 
    \path[->,font=\scriptsize] 
        (m-1-1) edge node[auto] {$\eta$} (m-1-3)
        (m-1-1) edge node[auto,swap] {$(p^*,id_A)$} (m-1-3)
        ;
\end{tikzpicture}

where $\\p^*:X\longrightarrow Hom(A,\{0,1\})$,

$x\longmapsto p_x:A\longrightarrow \{0,1\}$
such that $p_x(a)=1$ iff $x\models a$.

If possible let $p_x(\mathbf{0})=1$. Then we have $x\models 0$, which is a contradiction as $x\models 0$ for no $x\in X$. Hence $p_x(\mathbf{0})=0$. Also we have $p_x(a_1\wedge a_2)=1$ iff $x\models a_1\wedge a_2$ iff $x\models a_1$ and $x\models a_2$ iff $p_x(a_1)=1$ and $p_x(a_2)=1$ iff $p_x(a_1)\wedge p_x(a_2)=1$. Similarly it can be shown that $p_x(a_1\vee a_2)=p_x(a_1)\vee p_x(a_2)$ and $p_x(a_1\rightarrow a_2)=p_x(a_1)\rightarrow p_x(a_2)$. Hence for each $x\in X,$ $p_x:A\longrightarrow \{0,1\}$ is a Heyting homomorphism.

It may be observed that $x\models id_A(a)$ iff $x\models a$ iff $p_x(a)=1$ iff $(p^*(x))(a)=1$ iff $p^*(x)\models^* a$. Consequently we can conclude that  $(p^*,id_A):(X,\models,A)\longrightarrow \mathscr{S}(\mathscr{G}((X,\models,A))) $ is a continuous map of G$\ddot{o}$del algebraic I-topological system.

Let us define $\hat{f}$ as follows:
$(f_1,f_2):(X,\models ,A)\longrightarrow (Hom(B,\{0,1\}),\models_*,B)$
\\then $\hat{f}=f_2$.
Recall that $\mathscr{S}(\hat{f})=(\_\circ f_2,f_2)$.

It suffices to show that the triangle on the left commute, i.e., $(f_1,f_2)=\mathscr{S}(\hat{f})\circ \eta$. Now, $\mathscr{S}(\hat{f})\circ \eta=(\_\circ f_2,f_2)\circ (p^*,id_A)=((\_\circ f_2)\circ p^*,id_A\circ f_2)=((\_\circ f_2)\circ p^*, f_2)$. For any $x\in X,$ $f_1(x)=(\_\circ f_2)\circ p^*(x)=(\_\circ f_2)\circ p_x=p_x\circ f_2$. Consequently, for all $b\in B$, $f_1(x)(b)=1$ iff $f_1(x)\models^* b$ iff $x\models f_2(b)$ iff $p_x(f_2(b))=1$ iff $(p_x\circ f_2)(b)=1$ iff $((\_\circ f_2)\circ p_x)(b)=1$ iff $((\_\circ f_2)\circ p^*)(x)(b)=1$. Therefore $f_1=(\_\circ f_2)\circ p^*$. Hence $\eta(\equiv(p^*,id_A)):(X,\models ,A)\longrightarrow \mathscr{S}(\mathscr{G}((X,\models ,A)))$ is the unit, consequently $\mathscr{G}$ is the left adjoint to the functor $\mathscr{S}$.

Diagram of the co-unit of the above adjunction is as follows.

\begin{center}
\begin{tabular}{ l | r  } 
 $\mathbf{GA^{op}}$ & $\mathbf{GI-TopSys}$ \\
\hline
 {\begin{tikzpicture}[description/.style={fill=white,inner sep=2pt}] 
    \matrix (m) [matrix of math nodes, row sep=2.5em, column sep=2.5em]
    {\mathscr{G}(\mathscr{S}(A)) & & A  \\
         \mathscr{G}((Y,\models ,B)) \\ }; 
    \path[->,font=\scriptsize] 
        (m-1-1) edge node[auto] {$\xi(\equiv id_A)$} (m-1-3)
        (m-2-1) edge node[auto] {$f$} (m-1-1)
        (m-2-1) edge node[auto,swap] {$\mathscr{G}\hat{f}$} (m-1-3);
\end{tikzpicture}} &  {\begin{tikzpicture}[description/.style={fill=white,inner sep=2pt}] 
    \matrix (m) [matrix of math nodes, row sep=2.5em, column sep=2.5em]
    {& &\mathscr{S}(A) \\
        & &(Y,\models,B) \\ }; 
    \path[->,font=\scriptsize]

        (m-2-3) edge node[auto,swap] {$\hat{f}(\equiv\_\circ f)$} (m-1-3);
\end{tikzpicture}}\\
\end{tabular}
\end{center}
From the construction it can be easily shown that $\xi$ and $\eta$ are natural isomorphism and hence the theorem holds.
\end{proof}
\begin{theorem}\cite{CH}\label{GEF}
$\mathbf{GA}$ is dually equivalent with category of Esakia spaces whose order structure is a forest and Esakia morphisms ($\mathbf{FESA}$).
\end{theorem}
From Theorem \ref{GIGA} and Theorem \ref{GEF} the following theorem holds.
\begin{theorem}
$\mathbf{GI-TopSys}$ is equivalent to $\mathbf{FESA}$.
\end{theorem}
We can summarize our results by the following diagram.
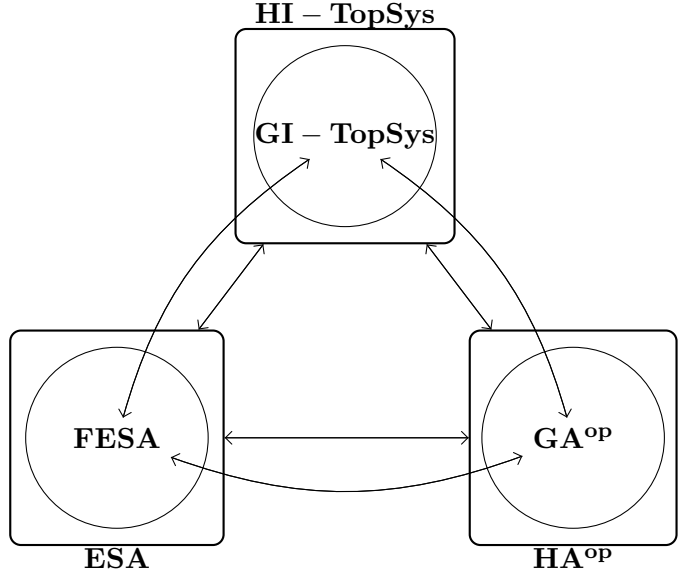
\def\firstcircle{(3,8cm) circle (1.2cm)}
\def\secondcircle{(4,8.65) circle (1.5cm)}

\def\firscircle{(0,4cm) circle (1.2cm)}
\def\seconcircle{(0,4.65) circle (1.5cm)}

\def\fircircle{(6,4cm) circle (1.2cm)}
\def\secocircle{(8,4.65) circle (1.5cm)}
\begin{center}
\begin{tikzpicture}
\centering
    \draw \firstcircle node[name=A] {$\mathbf{GI-TopSys}$};
    \draw \firscircle node[name=C] {$\mathbf{FESA}$};
    \draw \fircircle node[name=E] {$\mathbf{GA^{op}}$};
\node[draw, thick, rounded corners, inner xsep=0.3em, inner ysep=2.9em, fit=(A)] (B) {};
\node[draw, thick, rounded corners, inner xsep=1.8em, inner ysep=3.0em, fit=(C)] (D) {};
\node[draw, thick, rounded corners, inner xsep=1.8em, inner ysep=3.0em, fit=(E)] (F) {};
    \node (GFTS) at (3,9.6) {$\mathbf{HI-TopSys}$};
    \node (GFT) at (0,2.4) {$\mathbf{ESA}$};
    \node (GFr) at (6,2.4) {$\mathbf{HA^{op}}$};
\path[<-, font=\large, >=angle 90]
(A)edge [bend right=20] node[above] {} (C);  
\path[<-, font=\large, >=angle 90]
(C)edge [bend left=20] node[above] {} (A); 
\path[<-, font=\large, >=angle 90]
(B)edge [bend right=0] node[above] {} (D);  
\path[<-, font=\large, >=angle 90]
(D)edge [bend right=0] node[above] {} (B); 
\path[<-, font=\large, >=angle 90]
(E)edge [bend left=20] node[above] {} (C);  
\path[<-, font=\large, >=angle 90]
(C)edge [bend right=20] node[above] {} (E); 
\path[<-, font=\large, >=angle 90]
(D)edge [bend right=0] node[above] {} (F);  
\path[<-, font=\large, >=angle 90]
(F)edge [bend right=0] node[above] {} (D); 
\path[<-, font=\large, >=angle 90]
(A)edge [bend left=20] node[above] {} (E);  
\path[<-, font=\large, >=angle 90]
(E)edge [bend right=20] node[above] {} (A); 
\path[<-, font=\large, >=angle 90]
(B)edge [bend right=0] node[above] {} (F);  
\path[<-, font=\large, >=angle 90]
(F)edge [bend right=0] node[above] {} (B); 
\end{tikzpicture}
\end{center}
\section{Conclusion}
This paper suggest a new approach (new view) of representation of Heyting algebra as I-topological system. Moreover, relationship between the I-topological system and Esakia space and its particular case G\"{o}del space are shown. Connection of Kripke model with proposed system is indicated. It is expected that the proposed notion will play vital roles in the field of computer science and physics.
\subsection*{Acknowledgement}
The research work of the third author was supported by Dipartimento di Matematica, Universit\'a degli Studi di Salerno, the \emph{Indo-European Research Training Network in Logic} (IERTNiL) funded by the \emph{Institute of Mathematical Sciences, Chennai}, the \emph{Institute for Logic, Language and Computation} of the \emph{Universiteit van Amsterdam} and the \emph{Fakult\"at f\"ur Mathematik, Informatik und Naturwissenschaften} of the \emph{Universit\"at Hamburg} and Department of Science $\&$ Technology, Government of India under Women Scientist Scheme (reference no. SR/WOS-A/PM-52/2018). Part of this work was done when the second and the third author was visiting Dipartimento di Matematica, Universit\`a degli Studi di Salerno and they are thankful to the department for the warm hospitality.

\begin{footnotesize}
\begin{flushleft}
Antonio Di Nola\\
\textit{Universit\`{a} degli Studi di Salerno, Salerno, Italy}\\
\textit{E-mail:} adinola@unisa.it\\
\end{flushleft}
\end{footnotesize}

\begin{footnotesize}
\begin{flushleft}
Revaz Grigolia\\
\textit{Ivane Javakhishvili Tbilisi State University, Tbilisi, Georgia}\\
\textit{E-mail:} revaz.grigolia@tsu.ge, revaz.grigolia359@gmail.com\\
\end{flushleft}
\end{footnotesize}

\begin{footnotesize}
\begin{flushleft}
Purbita Jana\\
\textit{Indian Institute of Technology (IIT), Kanpur, India}\\
\textit{E-mail:} purbita@iitk.ac.in, purbita\_presi@yahoo.co.in\\
\end{flushleft}
\end{footnotesize}
Received ........
 
Accepted ........
\end{document}